\documentclass[12pt]{article}
\usepackage{amssymb,amsmath,latexsym,color}
\usepackage[mathscr]{eucal}
\usepackage{geometry}
\usepackage[auth-sc]{authblk}
\usepackage{tikz}
\usepackage{pgf}
\usetikzlibrary{arrows}
\usepackage{hyperref}

\newtheorem{theorem}{Theorem}[section]
\newtheorem{lemma}[theorem]{Lemma}

\newtheorem{sublemma}[theorem]{Sublemma}

\newtheorem{corollary}[theorem]{Corollary}

\newtheorem{remark}[theorem]{Remark}
\newenvironment{proof}{\noindent
  \textbf{Proof.}}{\hfill$\Box$\\}

\providecommand{\url}[1]{\texttt{#1}}

\providecommand{\doi}[1]{https://doi.org/#1}

\providecommand{\keywords}[1]{\textsc{\textsc{Keywords:}} #1}

\newcommand{\instr}[5]{\ensuremath{\hbox to 60 pt
    {${#1}$\hfil${#2}$\hfil$ \rightarrow
      $\hfil${#3}$\hfil${#4}$\hfil${#5}$}}}

\newcommand{\vp}{\ensuremath{\varphi}}

\newcommand{\nat}{\ensuremath{\mathbb{N}}}

\newcommand{\con}{\wedge}

\newcommand{\imp}{\rightarrow}

\newcommand{\bottom}{\perp}
\newcommand{\falsehood}{\ensuremath{\bot}}
\newcommand{\cm}[1]{\ensuremath{\sf{#1}}}

\newcommand{\mmodel}[1]{\ensuremath{\frak{#1}}}
\newcommand{\sat}[3]{\ensuremath{\frak{#1}, #2 \models #3}}
\newcommand{\notsat}[3]{\ensuremath{\mmodel{#1}, #2 \not\models #3}}

\newcommand{\psmodel}[1]{\ensuremath{\mathcal{M}^*}}
\newcommand{\pseudomodel}[1]{\ensuremath{\mathcal{M}^{**}}}

\newcommand{\sameas}{\ensuremath{\leftrightharpoons}}

\begin{document}

\title{Complexity of finite-variable fragments of propositional modal
  logics of symmetric frames\thanks{Pre-final version of the paper
    published in \emph{Logic Journal of the IGPL}, 27(1), 2019, pp. 60--68,
    DOI \doi{10.1093/jigpal/jzy018}}}

\author[1]{Mikhail Rybakov}

\author[2]{Dmitry Shkatov}

\affil[1]{Tver State University and University of the Witwatersrand,
  Johannesburg, \texttt{m\_rybakov@mail.ru}}

\affil[2]{University of the Witwatersrand, Johannesburg,
  \texttt{shkatov@gmail.com}}
\date{}
\maketitle

\begin{abstract}
  While finite-variable fragments of the propositional modal logic
  {\bf S5}---comp\-lete with respect to reflexive, symmetric, and
  transitive frames---are polynomial-time decidable, the restriction
  to finite-variable formulas for logics of reflexive and transitive
  frames yields fragments that remain ``intractable.''  The role of
  the symmetry condition in this context has not been investigated.
  We show that symmetry either by itself or in combination with
  reflexivity produces logics that behave just like logics of
  reflexive and transitive frames, i.e., their finite-variable
  fragments remain intractable, namely PSPACE-hard.  This raises the
  question of where exactly the borderline lies between modal logics
  whose finite-variable fragments are tractable and the rest.
  %Thus, what accounts
  %for peculiar computational properties of {\bf S5} is the combination
  %of all the three conditions---reflxivity, symmetry, and
  %transitivity---in its Kripke-style semantics.
\end{abstract}

\keywords{propositional modal logic, symmetric frames, finite-variable
  fragments, computational complexity}

\section{Introduction}

While the propositional modal logic {\bf S5}, which has Kripke-style
semantics in terms of reflexive, transitive, and symmetric frames, is
``computationally intractable''---namely, its satisfiability problem
is NP-complete---its $n$-variable fragments, for every $n \in \nat$,
are polynomial-time decidable.  In contrast, as originally shown
in~\cite{Halpern95} and further elaborated in~\cite{ChRyb03} (see
also~\cite{Hem01} and~\cite{Sve03}), most ``natural'' modal logics
whose Kripke frames are reflexive, transitive, or both, remain
intractable even if the number of propositional variables in their
languages is restricted to one (for all logics in the intervals [{\bf
  K}, {\bf GL}] and [{\bf K}, {\bf Grz}]) or zero (for all logics in
the interval [{\bf K}, {\bf K4}]).  It is, thus, interesting to see if
symmetry plays any role in making {\bf S5} stand apart from other
propositional modal logics in this regard.  More generally, does the
addition of the axiom of symmetry to a logic make its finite-variable
fragments easier to decide than the entire logic?  The role of
symmetry in this context has not been investigated in the
literature.

In this paper, we answer this question in the negative by showing that
all logics in the interval [{\bf K}, {\bf KTB}], where {\bf KTB} is
the propositional modal logic of reflexive and symmetric frames, have
PSPACE-hard single-variable fragments.  As a by-product, we prove that
logics {\bf KTB} and {\bf KB}, which is the propositional modal logic
of symmetric frames, can be embedded into their single-variable
fragments, which are, thus, as semantically expressive---from the
point of view of validity and (local) satisfiability---as the entire
logics.

The paper is organized as follows. In
section~\ref{sec:syntax-semantics}, we briefly recall the syntax and
semantics of the logics we consider in the present paper and establish
that all the logics in the interval [{\bf K}, {\bf KTB}] are
PSPACE-hard.  Then, in section~\ref{sec:finite-variable-fragments}, we
present our main results concerning single-variable fragments of
logics in [{\bf K}, {\bf KTB}]. We conclude in
section~\ref{sec:conclusion}.

\section{Preliminaries}
\label{sec:syntax-semantics}

%We briefly recall the syntax and semantics of the logics we consider
%in this paper.
The propositional modal language contains countably many propositional
variables $p_1, p_2, \ldots$, the Boolean constant $\bot$
(``falsehood''), the Boolean connective $\imp$, and the modal
connective $\Box$.  Other connectives, as well as formulas, are
defined as usual.  We also use the following abbreviations:
$$
\begin{array}{rclcrcl}
  \Box^0 \vp &=& \vp, & & \Box^{\leqslant 0} \vp &=& \vp, \\
  \Box^{n+1} \vp &=&  \Box \Box^n \vp, & & \Box^{\leqslant n + 1} \vp &=&
                                                                         \Box^{\leqslant n} \vp \con
                                                                         \Box^{n+1} \vp, \\
  \Box^+ \vp & = & \vp \con \Box \vp, &&
                                              %\Diamond^{\leqslant n}
%     \vp &=& \neg
%                                                      \Box^{\leqslant
%                                                      n} \neg \vp, \\
             \Diamond^n \vp &=& \neg \Box^n \neg \vp.
\end{array}
$$
%\noindent We denote by $\rm{\bf md}(\vp)$ the modal depth of the formula
%$\vp$, i.e., the maximal depth of nesting of modal operators $\Box$ in
%$\vp$.
%
A (Kripke) frame is a pair $\frak{F} = \langle W, R \rangle$, where
$W$ is a non-empty set (of worlds) and $R$ is a binary (accessibility)
relation on $W$.  A (Kripke) model is a pair
$\mmodel{M} = \langle \frak{F}, V \rangle$, where $\frak{F}$ is a
frame and $V$ is a valuation function assigning to every propositional
variable a subset of $W$; if \mmodel{M} has the form
$\langle \frak{F}, V \rangle$, we say that it is based on $\frak{F}$.
The satisfaction relation between models \mmodel{M}, worlds $w$, and
formulas $\vp$ is defined as follows: 
\begin{itemize}
\item \sat{M}{w}{p_i} \sameas\ $w \in V(p_i)$;
\nopagebreak[3]
\item \sat{M}{w}{\falsehood} never holds;
\nopagebreak[3]
\item \sat{M}{w}{\vp_1 \imp \vp_2} \sameas\ \sat{M}{w}{\vp_1} implies
  \sat{M}{w}{\vp_2};
\nopagebreak[3]
\item \sat{M}{w}{\Box \vp_1} \sameas\ \sat{M}{w'}{\vp_1} whenever
  $w R w'$.
\end{itemize}
Let $\frak{C}$ be a class of frames. A formula is valid on $\frak{C}$
if it is satisfied at every world of every model based on a frame from
$\frak{C}$. A formula is satisfiable in $\frak{C}$ if it is satisfied
at some world of some model based on a frame from $\frak{C}$.
%Then, a formula $\vp$ is
%satisfiable in $\frak{C}$ if it is satisfied at some state of some
%model based on a frame from $\frak{C}$.

A propositional modal logic is a set of formulas containing all
classical tautologies as well as the formula
$\Box (p \imp q) \imp (\Box p \imp \Box q)$ and closed under uniform
substitution, modus ponens, and necessitation.  Of particular interest
to us are the logics {\bf K}, which is the set of formulas valid on
all frames; {\bf KB}, which is the set of formulas valid on all frames
whose accessibility relation is symmetric; and {\bf KTB}, which is the
set of formulas valid on all frames whose accessibility relation is
reflexive and symmetric.  More generally, we will be concerned with
the interval [{\bf K, KTB}] of logics $L$ such that
${\bf K} \subseteq L \subseteq {\bf KTB}$.  This interval contains a
number of logics that have been, for various reasons, of interest to
logicians; examples include {\bf T}, which is the set of formulas
valid on reflexive frames; {\bf KB}; {\bf KDB}, which is the set of
formulas valid on symmetric and serial frames; and Hughes's
logic~\cite{Hughes90}.  If a logic $L$ is Kripke-complete, i.e.,
coincides with the set of formulas valid on some class $\frak{C}$ of
frames, we say that a formula $\vp$ is $L$-satisfiable if $\vp$ is
satisfiable in $\frak{C}$.

We say that a logic is PSPACE-hard (PSPACE-complete) if the problem of
membership in it is PSPACE-hard (PSPACE-complete); analogously for
fragments of logics.  In what follows, we rely on the statement as
well as the proof of the following:
%Kripke semantics is only one way of characterising a propositional
%modal logic.  To be able to speak of Kripke-incomplete logics---i.e.,
%More generally, we will be concerned with the interval
%[{\bf K, KTB}] of propositional modal logics $L$ such that
%${\bf K} \subseteq L \subseteq {\bf KTB}$.  This interval contains
%logics that are Kripke-incomplete, i.e.,
%logics that cannot be characterised as a set of formulas valid on a
%class of Kripke frames---we call a formula $\vp$ $L$-valid if
%$\vp \in L$ and $L$-consistent if $\neg \vp \notin L$.  If $L$ is
%Kripke-complete with respect to a class $\frak{C}$ of Kripke frames
%then $L$-validity amounts to validity on $\frak{C}$ and
%$L$-consistency amounts to satisfiability in $\frak{C}$.
%mean by $L$-validity the problem of
%whether a formula belongs to $L$.  We will also say that formula $\vp$
%is $L$-consistent if $\neg \vp \notin L$.  If $L$ is the set of
%formulas valid on a class $\frak{C}$ of Kripke frames, then
%$L$-consistency of $\vp$ amounts to the existence of a model
%\mmodel{M} based on a frame from $\frak{C}$ and a world $w$ such that
%\sat{M}{w}{\vp}.  This latter notion is also referred to as
%$L$-satisfiability.
\begin{theorem}
  \label{thr:Ladner}
  Let $L$ be a logic such that
  ${\bf K} \subseteq L \subseteq {\bf KTB}$.  Then, $L$ is
  {\rm PSPACE}-hard.
\end{theorem}

\begin{proof}
  The proof is a slight modification of Ladner's proof for logics
  between ${\bf K}$ and ${\bf S4}$ (Theorem 3.1 in~\cite{Ladner77};
  see also~\cite{BdeRV01}, Section 6.7), which proceeds by reduction
  from the set TQBF of true quantified Boolean formulas, known to be
  PSPACE-hard~\cite{SM73}.  Note that as PSPACE is closed under
  complementation, the complement of TQBF is also PSPACE-hard.  Since
  every quantified Boolean formula can be polynomially reduced to one
  in the prenex normal form, we may assume without a loss of
  generality that both TQBF and its complement only contain formulas
  in the prenex normal form.

  First, we define a polynomial-time computable translation $f$ from
  the set of quantified Boolean formulas in the prenex normal form to
  the set of modal formulas such that
  \begin{itemize}
  \item if $\theta \in \rm{TQBF}$, then $f(\theta)$ is {\bf
      KTB}-satisfiable;
  \item if $\theta \notin \rm{TQBF}$, then $f(\theta)$ is not {\bf
      K}-satisfiable.
  \end{itemize}
  %We use the fact that, for every $L$ in the statement of the Theorem,
  %the set of {\bf KTB}-consistent formulas is a subset of the set of
  %$L$-consistent formulas, which is a subset of the set of {\bf
  %  K}-consistent formulas.

  \noindent Let
  $\theta = \mathsf{Q}_1 p_1 \ldots \mathsf{Q}_m\, p_m\, \vp( p_1,
  \ldots, p_m)$,
  where $ \mathsf{Q}_1,\ldots,\mathsf{Q}_m \in \{\exists, \forall\}$
  and $\vp( p_1, \ldots, p_m)$ is a propositional formula containing
  no variables other than $p_1, \ldots, p_n$.  Let
  $q_0, q_1, \ldots, q_m$ be propositional variables not in $\theta$.
  Then, $f(\theta)$ is a conjunction of the following formulas:
  \begin{itemize}
  \item $q_0$;
  \item
    $\Box^{\leqslant m} \bigwedge\limits_{i = 0}^m ( q_i \imp
    \bigwedge\limits_{j \ne i} \neg q_j)$;
  \item
    %$\Box^{\leqslant m-1} \bigwedge\limits_{i = 0}^{m-1} ( q_i \imp
    %\Diamond q_{i+1})$;
    $\Box^{\leqslant m-1} \bigwedge\limits_{\{i\,{:}\,\mathsf{Q}_i =
      \exists \}} ( q_{i-1} \imp \Diamond q_{i})$;
  \item
    $\Box^{\leqslant m-1} \bigwedge\limits_{\{i\,{:}\,\mathsf{Q}_i = \forall \}} (q_{i-1} \imp \Diamond (q_{i}
    \con p_{i}) \con \Diamond (q_{i} \con \neg p_{i}))$;
  \item
    $\Box^{\leqslant m - 1} \bigwedge\limits_{i = 1}^{m-1} ( q_i \imp
    \bigwedge\limits_{j \leqslant i} (p_j \imp \Box (q_{i+1} \imp
    p_j)) \con \bigwedge\limits_{j \leqslant i} (\neg p_j \imp \Box
    (q_{i+1} \imp \neg p_j)))$;
  \item $\Box^{m} (q_m \imp \vp)$.
  \end{itemize}

  \noindent Note that only the second-to-last formula is substantively
  different from the formulas used in~\cite{Ladner77}.  Suppose that
  $\theta$ is true, and thus, there exists a quantifier tree $T$
  witnessing its truth.  We use $T$ to define a {\bf KTB}-model
  satisfying $f(\theta)$.  Let $W$ be the set of nodes of $T$ and $R$
  be the symmetric and reflexive closure of the ``daughter-of''
  relation of $T$.  Thus, $\langle W, R \rangle$ is a {\bf KTB}-frame.
  It remains to define the valuation.  Let $q_i$ be true precisely at
  the nodes of level $i$ (where the root is a node of level $0$), let
  $p_i$ be true at a node of level $j \geqslant i$ if, and only if,
  the substitution of truth values for a variable of $\theta$
  connected to that node, or to the node of level $i$ on the same
  branch of $T$, returns ``true'' for $p_i$, and let $p_i$ be false at
  all nodes of levels $j < i$.  It is then straightforward to check
  that $f(\theta)$ is satisfied at the root of $T$.  That falsehood
  of $\theta$ implies that $f(\theta)$ is not {\bf K}-satisfiable is
  argued exactly as in Ladner's proof~\cite{Ladner77}.

  Now, let $L$ be a logic such that
  ${\bf K} \subseteq L \subseteq {\bf KTB}$.  If
  $\theta \notin \rm{TQBF}$, then $\neg f(\theta) \in {\bf K}$ and,
  hence, $\neg f(\theta) \in L$.  Conversely, if
  $\theta \in \rm{TQBF}$, then $\neg f(\theta) \notin {\bf KTB}$ and,
  hence, $\neg f(\theta) \notin L$.  Thus, the translation
  $t(\theta) = \neg f (\theta)$ reduces the complement of TQBF, which
  is PSPACE-hard, to $L$.  Therefore, $L$ is PSPACE-hard.
\end{proof}

%\begin{remark}
%  \label{rem:Ladner}
%  Since the translation $t$ used in the proof of
%  Theorem~\ref{thr:Ladner} was defined in the same way for every logic
%  in {\rm{[{\bf K}, {\bf KTB}]}}, it follows that the PSPACE-hard set
%  $F = \{ \vp \mid \vp = f(\theta) \mbox { for some valid QBF } \theta \}$
%  is a subset of the set of $L$-consistent formulas for every
%  $L \in [{\bf K}, {\bf KTB}]$.
%\end{remark}

As there exist polynomial-space algorithms for deciding
satisfiability, and thus validity, for {\bf KB}, {\bf KDB}, and {\bf
  KTB} (see, e.g., \cite{Gore99}), these logics are PSPACE-complete.

\section{Complexity of finite-variable fragments }
\label{sec:finite-variable-fragments}

We now show, using a suitable modification of Halpern's
technique~\cite{Halpern95} (see also~\cite{RSh18}), that
single-variable fragments of all logics in the interval [{\bf K, KTB}]
are PSPACE-hard.  In the course of the proof we establish that logics
{\bf KB} and {\bf KTB} can be effectively embedded into their
single-variable fragments.

Let $\vp$ be an arbitrary modal formula.  Assume that $\vp$ only
contains propositional variables $p_1, \ldots, p_n$. First,
recursively define the translation $\cdot'$ as follows:
\begin{center}
  \begin{tabular}{llll}
    ${p_i}'$ & = & $p_i,$ $\mbox{~~where~} i \in \{1, \ldots, n \}$; \\
    $\bottom'$ & = & $\bottom$; & \\
    $(\phi \imp \psi)'$ & = & $\phi' \imp \psi'$; & \\
    $(\Box \phi)'$ & = & $\Box (p_{n+1} \imp \phi')$. &
\end{tabular}
\end{center}
Second, put $$\widehat{\vp} = p_{n+1} \con \vp'. $$

\noindent Notice that $\vp$ is equivalent to
$\widehat{\vp} (p_{n+1} / \top)$ in {\bf K} and, hence, in {\bf KTB}.
%$$ \Theta = p_{n+1} \con \Box^{\rm{\leqslant \bf md}(\vp) - 1} (\Diamond p_{n+1} \imp
%p_{n+1}). $$ Finally, let

\begin{lemma}
  \label{lem:pn+1}
  Let $L \in \{ \bf{K}, \bf{KTB} \}$.  If $\widehat{\vp}$ is
  $L$-satisfiable, then it is satisfiable in a model based on a frame
  for $L$ where $p_{n+1}$ is true at every world.
\end{lemma}

\begin{proof}
  Suppose that \sat{M}{w_0}{\widehat{\vp}} for some model \mmodel{M}
  and some world $w_0$.  Consider the submodel \mmodel{M'} of
  \mmodel{M} that consists of worlds where $p_{n+1}$ is true.  As
  \sat{M}{w_0}{p_{n+1}}, the set of worlds of \mmodel{M'} is
  non-empty.  It is straightforward to check both that $\mmodel{M}'$
  is based on a frame for $L$ and that \sat{M'}{w_0}{\widehat{\vp}}.
\end{proof}

\begin{lemma}
  \label{lem:varphi-truth}
  Let $L \in \{ \bf{K}, \bf{KTB} \}$. Then, $\varphi$ is
  $L$-satisfiable if, and only if, $\widehat{\vp}$ is $L$-satisfiable.
\end{lemma}

\begin{proof}
  Suppose that $\sat{M}{w_0}{\vp}$.  To obtain a model satisfying
  $\widehat{\vp}$, make $p_{n+1}$ true at every world of \mmodel{M}.
  Conversely, suppose that $\sat{M}{w_0}{\widehat{\vp}}$.  In view of
  Lemma~\ref{lem:pn+1}, we may assume that $p_{n+1}$ is universally
  true in \mmodel{M}.  As $\vp$ is equivalent to
  $\widehat{\vp} (p_{n+1} / \top)$, it follows that
  $\sat{M}{w_0}{\vp}$.
%  Suppose $\widehat{\vp}$ is not satisfiable, i.e.,
%  $\neg \widehat{\vp} \in L$; thus,
%  $\neg \widehat{\vp} (p_{n+1} / \top) \in L$.  As
%  $\widehat{\vp} (p_{n+1} / \top) \equivalence \vp \in L$, we have
%  $\neg \vp \in L$, as required.
%
%  Suppose that $\widehat{\vp}$ is satisfiable. In particular, let
%  $\sat{M}{w_0}{\widehat{\vp}}$ for some model $\mmodel{M}$ based on a
%  frame for $L$ and some $w_0$ in \mmodel{M}. Define $\mmodel{M}'$ to
%  be the smallest submodel of $\mmodel{M}$ such that
%  \begin{itemize}
%  \item $w_0$ is in $\mmodel{M'}$;
%  \item if $x$ is in $\mmodel{M'}$, $x R y$, and
%    \sat{M}{y}{p_{n+1}}, then $y$ is also in $\mmodel{M'}$.
%  \end{itemize}
%  Notice that $p_{n+1}$ is universally true in $\mmodel{M}'$.  It is
%  straightforward to show that $\sat{M'}{w_0}{\vp}$. As $\mmodel{M}'$
%  is based on a frame for $L$, we conclude that $\vp$ is $L$-satisfiable.
%  for every subformula $\psi$ of $\vp$,
%  we have $\sat{M}{s}{\psi'}$ if, and only if, $\sat{M'}{s}{\psi}$.
%  As $\sat{M}{s}{\vp'}$, this gives us $\sat{M'}{s}{\vp}$.  The claim
%  follows as .
\end{proof}
%
%\begin{remark}
%  \label{remark:pn+1}
%  It follows from the proof of Lemma~\ref{lem:varphi-truth} that if
%  $\widehat{\vp}$ is satisfiable, then it is satisfiable in a model
%  where $p_{n+1}$ is universally true.  Indeed, if $\widehat{\vp}$ is
%  satisfiable, then $\vp$ is satisfiable in a model where $p_{n+1}$ is
%  universally true. The claim then follows from the fact that $\vp$ is
%  equivalent to $\widehat{\vp} (p_{n+1} / \top)$.
%\end{remark}

Now, consider the following class $\cm{M}$ of finite models.  For
every $k \in \{1, \ldots, n+1 \}$, the class $\cm{M}$ contains a model
$\frak{M}_k$, depicted in Figure~\ref{fig-model-M-k}, that looks as
follows.  For brevity, we call some worlds $p$-worlds; if a world is
not a $p$-world, we call it a $\bar{p}$-world.  The model $\frak{M}_k$
is a chain of worlds whose root, $r_k$, is a $p$-world.  The root is
part of a pattern of worlds, described below, which is succeeded by
three final $p$-worlds.  The pattern looks as follows: a single
$p$-world is followed by $2i + 1$ $\bar{p}$-worlds, for
$1 \leqslant i \leqslant k$.  Thus, the chain looks as follows: the
root (a $p$-world), then three $\bar{p}$-worlds, then a $p$-world,
then five $\bar{p}$-worlds, then a $p$-world, \ldots, then a
$p$-world, then $2k + 1$ $\bar{p}$-worlds, then three $p$-worlds.  The
accessibility relation $R_k$ between the worlds of $\frak{M}_k$ is
both reflexive and symmetric.  To complete the definition of
$\frak{M}_k$, we define the propositional variable $p$ to be true at
exactly the $p$-worlds.

\begin{figure}
\centering
\begin{tikzpicture}[scale=0.68]

%points
\foreach \x in {1,2,3,4,5,6,7,8,9,10,11,12,15,16,17,18}
\draw[thick] (\x,1) circle [radius=2.5pt];

\foreach \x in {1,2,3,4,5,6,7,8,9,10,11,12,14,15,16,17}
\draw [-,  shorten >= 2.5pt, shorten <= 2.5pt] (\x,1) -- (\x+1,1);

\draw (13.5,1) node {$\ldots$};

\draw (1-0.15,0.5) node {$r_k$};
\draw (3-0.15,0.5) node {$c_1^k$};
\draw (8-0.15,0.5) node {$c_2^k$};

\draw (1.75, 1.15) -- (1.75, 1.25) -- (4.25, 1.25) -- (4.25, 1.15);
\draw (5.75, 1.15) -- (5.75, 1.25) -- (10.25, 1.25) -- (10.25, 1.15);
\draw (11.75, 1.15) -- (11.75, 1.25) -- (13.00, 1.25);
\draw (14.00, 1.25) -- (15.25, 1.25) -- (15.25, 1.15);
\foreach \x in {1,5,11,16,17,18}
\draw (\x,1.5) node {$\phantom{\neg}p$};

\foreach \x in {3,8,12.5,14.5}
\draw (\x,1.5) node {$\neg p$};
\end{tikzpicture}
\caption{Model $\frak{M}_k$}
\label{fig-model-M-k}
\end{figure}
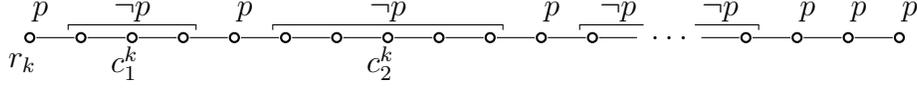

Before proceeding, we prove a lemma about the models in \cm{M}. Given
a model $\mmodel{M}_k$, denote by $c_i^k$, for
$i \in \{ 1, \ldots, k \}$, the ``middle'' world of the chain of
$2i + 1$ $\bar{p}$-worlds preceded and succeeded by $p$-worlds; see
Figure~\ref{fig-model-M-k}. Also, let
%(for
%example, $c^k_1$ is the second world in the chain of three
%$\bar{p}$-worlds between $r_k$ and the $p$-world reachable from $r_k$
%in four steps
\begin{center}
  \begin{tabular}{lll}
    $\varepsilon_i$ & = & $\Box^{\leqslant i}\neg p \con \Diamond^{i+1}
                       p$, where $i \in \nat$.
\end{tabular}
\end{center}

\begin{lemma}
  \label{lem:epsilons}
  Let $x$ be a world of $\mmodel{M}_k$ that lies between $r_k$ and
  $c^k_i$, for some $i \leqslant k$ (i.e., $c^k_i$ cannot be reached
  from $r_k$ by consecutive steps along $R_k$ without passing
  through~$x$). Then, $\mmodel{M}_k, x \models \varepsilon_i$ if, and
  only if, $x = c^k_i$.
\end{lemma}

\begin{proof}
  Straightforward.
\end{proof}

We now define formulas we use to simulate the propositional variables
of $\widehat{\vp}$.  First, inductively define, for every
$k \in \{1, \ldots, n+1 \}$, the following sequence of formulas:
\begin{center}
  \begin{tabular}{lll}
    $\delta$ & = & $\Box^+ p$; \\
    %$\delta_k^k$ & = & $\Box^{\leqslant k} \neg p \con
    %                      \Diamond^{k+1} p \con \Diamond^{k+2}\delta$;
    %\\
    $\delta_k^k$ & = & $\varepsilon_k \con \Diamond^{k+2}\delta$; \\
    $\delta_{i}^k$ & = & $\varepsilon_i \con
                         \Diamond^{2i+3}\delta^k_{i+1}$, where
                         $1\leqslant i < k$.
    %$\delta_{i}^k$ & = & $\Box^{\leqslant i}\neg p\con
     %                       \Diamond^{i+1} p\con
     %                       \Diamond^{2i+3}\delta^k_{i+1}$, where
      %                   $1\leqslant i < k$.

\end{tabular}
\end{center}

\noindent Next, let, for every $k \in \{1, \ldots, n+1 \}$,
$$
\alpha_k = p \con \Diamond^2 \delta^k_1
%\con \Diamond^2\Box^+
%\neg p.
$$
\noindent and $$\beta_k = \neg p \con \Diamond \alpha_k.$$

\noindent Let $\sigma$ be a (substitution) function that, given a formula
$\psi$, replaces all occurrences of $p_i$ in $\psi$ by $\beta_i$,
where $1 \leqslant i \leqslant n + 1$. Finally, define
$$\varphi^* = \sigma(\widehat{\vp})$$
to produce a single-variable formula $\varphi^*$.
\begin{lemma}
  \label{lem:main_lemma}
  Let $L \in \{ \bf{K}, \bf{KTB} \}$. Then, $\varphi$ is
  $L$-satisfiable if, and only if, $\varphi^*$ is $L$-satisfiable.
\end{lemma}

\begin{proof}
  Suppose that $\varphi$ is not $L$-satisfiable.  Then, in view of
  Lemma~\ref{lem:varphi-truth}, $\widehat{\vp}$ is not
  $L$-satisfiable; hence, $\neg \widehat{\vp} \in L$.  Since $L$ is
  closed under substitution, $\neg \varphi^* \in L$, and so
  $\varphi^*$ is not $L$-satisfiable.

  Suppose that $\vp$ is $L$-satisfiable.  Then, in view of
  Lemmas~\ref{lem:pn+1} and~\ref{lem:varphi-truth},
  $\sat{M}{w_0}{\widehat{\varphi}}$ for some
  $\mmodel{M} = \langle W, R, V \rangle$, such that
  $\langle W, R \rangle$ is a frame for $L$ and $p_{n+1}$ is true at
  every $w \in W$, and for some $w_0 \in W$.  (Recall that
  $\widehat{\varphi}$ only contains variables $p_1, \ldots, p_{n+1}$.)
  Define model \mmodel{M'} as follows. Attach to \mmodel{M} all the
  models from \cm{M}; then, for every $x$ in $\mmodel{M}$, put
  $x R' r_m$ and $r_m R' x$, where $r_m$ is the root of
  $\mmodel{M}_m \in \cm{M}$, exactly when \sat{M}{x}{p_m}.  Notice
  that $r_{n+1}$ is accessible in \mmodel{M'} from every $x \in W$.
  Finally, make $p$ true at exactly those worlds of the attached
  models where it was true, and make it false at every world in $W$.
  Notice that $\mmodel{M}'$ is based on a frame for $L$.

  To conclude the proof, it suffices to show that
  \sat{M'}{w_0}{\varphi^*}.  To that end, we first prove two auxiliary
  Sublemmas:

  \begin{sublemma}
    \label{sublemma:e}
    Let $x$ be a world of $\mmodel{M}'$ that lies between the root
    $r_k$ of the attached model $\mmodel{M}_k$ and the world $c^k_i$
    of $\mmodel{M}_k$, for some $i \leqslant k$ (i.e., $c^k_i$ cannot
    be reached from $r_k$ by consecutive steps along $R'$ without
    passing through $x$). Then, $\mmodel{M}', x \models \varepsilon_i$
    if, and only if, $x = c^k_i$.
  \end{sublemma}

  \begin{proof}
    Straightforward, using Lemma~\ref{lem:epsilons}.
  \end{proof}

  \begin{sublemma}
    \label{sublem:1}
    Let $x \in W$ and let $\sat{M'}{x}{\Diamond \alpha_k}$.  Then,
    $x R' r_k$.
  \end{sublemma}

  \begin{proof}
    Since $\sat{M'}{x}{\Diamond \alpha_k}$, so $x R' y$ and
    $\sat{M'}{y}{\alpha_k}$, for some $y$ in $\mmodel{M}'$.  We show
    that $y = r_k$.  Since $\sat{M'}{y}{p}$, clearly $y \notin W$, and
    thus $y$ is the root $r_m$ of some $\mmodel{M}_m$. As
    $\sat{M'}{y}{\Diamond^2 \delta^k_1}$, we can reach from $y$ in two
    $R'$-steps a world $y_1$ such that $\sat{M'}{y_1}{\varepsilon_1}$.
    Since $w R' r_{n+1}$ holds for every $w \in W$, and thus
    $\notsat{M'}{w}{\Box \neg p}$ for every $w \in W$, we know that
    $y_1 \notin W$, so $y_1$ belongs to one of the attached models
    $\mmodel{M}_j$.  In two $R'$-steps, we cannot go past $c^j_1$ for
    any $j$ and can only reach $c^j_1$ if $j = m$; hence, due to
    Sublemma~\ref{sublemma:e}, $y_1 = c^m_1$.  Since
    $\sat{M'}{c^m_1}{\Diamond^5( \varepsilon_2 \con \delta^k_2)}$, we
    can reach from $c^m_1$ in five $R'$-steps a world $y_2$ such that
    $\sat{M'}{y_2}{\varepsilon_2}$.  As $\notsat{M'}{w}{\Box \neg p}$
    for every $w \in W$, we know that $y_2 \notin W$.  In five
    $R'$-steps, we cannot go past $c^j_2$ for any $j$ and can only
    reach $c^j_2$ if $j = m$; hence, due to Sublemma~\ref{sublemma:e},
    $y_2 = c^m_2$, and so
    $\sat{M'}{c^m_2}{\Diamond^7( \varepsilon_3 \con \delta^k_3)}$.  We
    can now repeat the argument without worrying about the possibility
    of satisfying further formulas due to the presence in
    $\mmodel{M}'$ of the worlds outside of $\mmodel{M}_m$, as we
    cannot step outside of $\mmodel{M}_m$, starting from $c^m_2$, in
    seven steps.  By inductively repeating the argument $m$ times, we
    arrive at the world $c^m_m$ such that
    $\sat{M'}{c^m_m}{\Diamond^{k+2} \delta}$, which can only happen if
    $m = k$.  Thus, all along we have been evaluating the formulas in
    $\mmodel{M}_k$, and thus $y = r_k$, as required.
  \end{proof}

  Now, we proceed with the proof of the main Lemma.

  Recall that
  $\vp^\ast = \sigma(\widehat{\vp}) = \sigma(p_{n+1} \con \vp') =
  \beta_{n+1} \con \sigma(\vp')$.
  It is easy to check that \sat{M'}{w_0}{\beta_{n+1}}.  It then
  suffices to show that \sat{M}{x}{\psi'} if, and only if,
  \sat{M'}{x}{\sigma (\psi')}, for every subformula $\psi$ of
  $\varphi$ and every $x \in W$.  This can be done by induction on
  $\psi$.%; we only
  % consider the base case, leaving the rest to the reader.

  For the base case, assume that \sat{M'}{x}{\beta_i}; in particular,
  \sat{M'}{x}{\Diamond \alpha_i}. Then, due to
  Sublemma~\ref{sublem:1}, $x R' r_i$, and therefore \sat{M}{x}{p_i}
  by definition of $\mmodel{M}'$.  The other direction is
  straightforward.  The Boolean cases are also straightforward.

  Let $\psi = \Box \chi$. Assume that \notsat{M'}{x}{\Box (\beta_{n+1}
    \imp \sigma(\chi'))}.  Then, $x R' y$, as well as
  \sat{M'}{y}{\beta_{n+1}} and \notsat{M'}{y}{\sigma(\chi')}, for some
  $y$ in $\mmodel{M}'$.  In particular, \sat{M'}{y}{\neg p}; thus, $y$
  cannot be the root of any of the attached models.  Therefore,
  $y \in W$ and the inductive hypothesis is applicable; this gives us
  \notsat{M}{x}{\Box (p_{n+1} \imp \chi')}, as desired. The other
  direction is straightforward, using the converse of
  Sublemma~\ref{sublem:1}.
\end{proof}

Given a formula $\vp$, let
$$
e(\vp) = \neg ((\neg \vp)^\ast).
$$

\begin{theorem}
  \label{thr:main-theorem}
  Let $L \in \{{\bf K}, {\bf KTB}\}$. Then, there exists a
  polynomial-time mapping that embeds $L$ into its single-variable
  fragment.
\end{theorem}

\begin{proof}
  Take the mapping $e$ defined above.
\end{proof}

\begin{remark}
  Notice that Lemma~\ref{lem:main_lemma} and
  Theorem~\ref{thr:main-theorem} apply to the logic {\bf KB}, as well.
  We did not mention {\bf KB} in their statements as this is not
  required for the proof of our main result, Theorem~\ref{cor:main}.
\end{remark}

\begin{theorem}
  \label{cor:main}
  Let $L$ be a logic in the interval $[{\bf K}, {\bf KTB}]$.  Then,
  the single-variable fragment of $L$ is {\rm PSPACE}-hard.
\end{theorem}

\begin{proof}
  We reduce the PSPACE-hard complement of the set TQBF of true
  quantified Boolean formulas to the single-variable fragment of $L$.
  Let $\theta \notin \rm{TQBF}$; then, $t(\theta) \in {\bf K}$, where
  $t$ is the translation defined in the proof of
  Theorem~\ref{thr:Ladner}; hence, in view of
  Theorem~\ref{thr:main-theorem}, $e( t(\theta)) \in {\bf K}$, and
  thus $ e( t(\theta)) \in L$.  Let, on the other hand,
  $\theta \in \rm{TQBF}$; then, as shown in the proof of
  Theorem~\ref{thr:Ladner}, $t(\theta) \notin {\bf KTB}$; hence, in
  view of Theorem~\ref{thr:main-theorem},
  $e( t(\theta)) \notin {\bf KTB}$, and thus
  $ e( t(\theta)) \notin L$.  Thus, the polynomial-time computable
  translation $g(\theta) = e( t(\theta))$ reduces the complement of
  TQBF to the single-variable fragment of $L$; the statement of the
  Theorem follows.
\end{proof}

Theorem~\ref{cor:main} implies that the single-variable fragments of
all PSPACE-complete logics in $[{\bf K}, {\bf KTB}]$ are {\rm
  PSPACE}-complete; in particular, we have the following:

\begin{corollary}
  The single-variable fragments of {\bf T}, {\bf KB}, {\bf KDB}, and
  {\bf KTB} are {\rm PSPACE}-complete.
\end{corollary}

\noindent Note that the PSPACE-completeness of the single-variable
fragment of {\bf T} has been established in~\cite{Halpern95}.
%Notice that, for both {\bf KB} and {\bf KTB},
%the variable-free fragment is polynomial-time computable, as for {\bf
%  KB} every variable-free formula is equivalent to one of $\top$,
%$\bot$, $\Diamond \top$ and $\Box \bot$, while for {\bf KTB} every
%variable-free formula is equivalent either $\top$ or $\bot$.

\section{Conclusion}
\label{sec:conclusion}

We have shown that when it comes to their computational properties,
the modal logics of symmetric, as well as of reflexive and symmetric
frames, behave in the same way as the logics of transitive, as well as
of reflexive and transitive, frames---they remain intractable, namely,
PSPACE-hard, when their languages are restricted to only one
propositional variable.
%We notice that if we take the other remaining
%combination of the properties of {\bf S5}-frames, i.e., symmetry and
%transitivity, we obtain a logic that has the computational properties
%similar to {\bf S5}---its finite-variable fragments are
%polynomial-time decidable, as this logic is characterized by frames
%that look like {\bf S5}-frames that might additionally contain
%irreflexive worlds.

Adding the axiom of symmetry to the logic of reflexive and transitive
frames, i.e., {\bf S4}, is not the only way of arriving at {\bf
  S5},---it can also be obtained, inter alia, by adding the axiom of
Euclideanness, $\neg \Box p \imp \Box \neg \Box p$, to {\bf T}.  The
role of Euclideanness is well understood in the context of the present
inquiry,---it has been shown in~\cite{NTh75} that every extension of
the logic of Euclidean frames, {\bf K5}, is locally tabular;
therefore, any finite-variable fragment of such a logic is
polynomial-time decidable.

\begin{figure}
\centering
\begin{tikzpicture}[scale=1.25]

\coordinate (K)    at (0.00,0.00);
\coordinate (KB)   at (4.20,0.00);
\coordinate (T)    at (0.00,3.75);
\coordinate (KTB)  at (4.20,3.75);
\coordinate (K4)   at (1.30,1.35);
\coordinate (KB5)  at (4.75,1.35);
\coordinate (S4)   at (1.30,4.60);
\coordinate (S5)   at (4.75,4.60);
\coordinate (KD)   at (0.00,2.05);
\coordinate (KDB)  at (4.20,2.05);
\coordinate (KD4)  at (1.30,3.05);
\coordinate (KD45) at (3.25,3.05);
\coordinate (K45)  at (3.25,1.35);
\coordinate (K5)   at (1.90,0.80);
\coordinate (KD5)  at (1.90,2.65);

%lines
\draw [-,  shorten >= 1.5pt, shorten <= 1.5pt, color = gray]
(K) -- (KB) -- (KTB) -- (T) -- (K) -- (K4) -- (S4) -- (S5) -- (KB5) -- (K4);
\draw [-,  shorten >= 1.5pt, shorten <= 1.5pt, color = gray]
(KDB) -- (KD) -- (KD4) -- (KD45) -- (K45) -- (K);
\draw [-,  shorten >= 1.5pt, shorten <= 1.5pt, color = gray]
(KD45) -- (S5);
\draw [-,  shorten >= 1.5pt, shorten <= 1.5pt, color = gray]
(T) -- (S4);
\draw [-,  shorten >= 1.5pt, shorten <= 1.5pt, color = gray]
(KTB) -- (S5);
\draw [-,  shorten >= 1.5pt, shorten <= 1.5pt, color = gray]
(KB) -- (KB5);
\draw [-,  shorten >= 1.5pt, shorten <= 1.5pt, color = gray]
(KD) -- (KD45);
\draw [-,  shorten >= 1.5pt, shorten <= 1.5pt, color = gray]
(K5) -- (KD5);

\node [below left ] at (K)    {${\bf K}$}   ;
\node [below right] at (KB)   {${\bf KB}$}  ;
\node [above left ] at (T)    {${\bf T}$}   ;
\node [above right] at (KTB)  {${\bf KTB}$} ;
\node [above left ] at (K4)   {${\bf K4}$}  ;
\node [above right] at (KB5)  {${\bf KB5}$} ;
\node [above left ] at (S4)   {${\bf S4}$}  ;
\node [above right] at (S5)   {${\bf S5}$}  ;
\node [above left ] at (KD)   {${\bf KD}$}  ;
\node [above right] at (KDB)  {${\bf KDB}$} ;
\node [above left ] at (KD4)  {${\bf KD4}$} ;
\node [above right] at (KD45) {${\bf KD45}$};
\node [below right] at (K45)  {${\bf K45}$} ;
\node [below right] at (K5)   {${\bf K5}$}  ;
\node [below right] at (KD5)  {${\bf KD5}$} ;

%points

\shadedraw [shading=ball,ball color = gray] (K)    circle [radius=1.5pt];
\shadedraw [shading=ball,ball color = gray] (KB)   circle [radius=1.5pt];
\shadedraw [shading=ball,ball color = gray] (T)    circle [radius=1.5pt];
\shadedraw [shading=ball,ball color = gray] (S4)   circle [radius=1.5pt];
\shadedraw [shading=ball,ball color = gray] (S5)   circle [radius=1.5pt];
\shadedraw [shading=ball,ball color = gray] (KD4)  circle [radius=1.5pt];
\shadedraw [shading=ball,ball color = gray] (KTB)  circle [radius=1.5pt];
\shadedraw [shading=ball,ball color = gray] (KDB)  circle [radius=1.5pt];
\shadedraw [shading=ball,ball color = gray] (K4)   circle [radius=1.5pt];
\shadedraw [shading=ball,ball color = gray] (K45)  circle [radius=1.5pt];
\shadedraw [shading=ball,ball color = gray] (KD45) circle [radius=1.5pt];
\shadedraw [shading=ball,ball color = gray] (KB5)  circle [radius=1.5pt];
\shadedraw [shading=ball,ball color = gray] (KD5)  circle [radius=1.5pt];
\shadedraw [shading=ball,ball color = gray] (K5)   circle [radius=1.5pt];
\shadedraw [shading=ball,ball color = gray] (KD)   circle [radius=1.5pt];

\end{tikzpicture}
\caption{Cube of modal logics}
\label{fig:cube}
\end{figure}
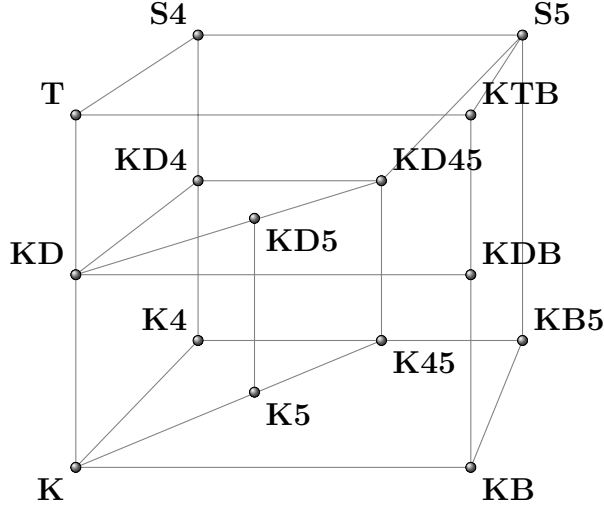

Thus, we have a good understanding of the role played by various
properties of Kripke frames of most interest to ``traditional''
logicians (reflexivity, seriality, symmetry, transitivity, and
Euclideanness)---represented by logics included in the ``cube of modal
logics''~\cite{Garson16}, see Figure~\ref{fig:cube}---in the
computational behaviour of the finite-variable fragments of the
corresponding logics: while Euclideanness---as well as symmetry
combined with transitivity, which imply Euclideanness---make such
fragments ``tractable'', reflexivity, symmetry, transitivity, and
seriality by themselves---as well as transitivity and symmetry
combined either with reflexivity or seriality---do not have this
effect (seriality, along with reflexivity and transitivity, has been
considered in~\cite{ChRyb03}).

This raises the more general question of where the borderline lies
between, on the one hand, logics that behave like those described
in~\cite{Halpern95}, \cite{Hem01}, \cite{Sve03}, \cite{ChRyb03}, and
in this paper, i.e., whose finite-variable fragments remain
intractable, and on the other, those that behave like {\bf S5}, i.e.,
whose finite-variable fragments are simpler than entire logics (in all
the cases known in the literature, this amounts to having
polynomial-time decidable finite-variable fragments).  It is clear
that the answer is not directly linked to the complexity of the logic
in question---as shown in~\cite{ChRyb03}, satisfiability for single-
and two-variable fragments of such logics as {\bf S4.3}, {\bf GL.3},
and {\bf Grz.3}, whose satisfiability problem is \rm{NP}-complete, is
also \rm{NP}-complete. While the borderline between the NP-hard and
the PSPACE-hard in modal logic has received attention in the
literature (see, for example, \cite{HalpernRego07}), this question has
not, as far as we know, been so far addressed.

\end{document}